\documentclass[11pt]{article}
\usepackage[]{times}
\usepackage{amsmath}
\usepackage{paralist}
\usepackage{fullpage}
\usepackage{amssymb}
\usepackage{mathrsfs}
\usepackage{bm}
\usepackage{graphicx}
\usepackage{natbib}
\usepackage{graphicx}
\usepackage{wrapfig}
\usepackage{epsfig}
\usepackage{url}
\usepackage{fullpage}
\usepackage{psfrag}
\usepackage{verbatim}
\usepackage{amsthm}
\usepackage[]{times}
\usepackage{amsmath}
\usepackage{paralist}
\usepackage{fullpage}
\usepackage{amssymb}

\usepackage[sans]{dsfont}

\usepackage{mathrsfs}
\usepackage{bm}
\usepackage{graphicx}
\usepackage{natbib}
\usepackage{graphicx}
\usepackage{wrapfig}
\usepackage{epsfig}
\usepackage{url}
\usepackage{fullpage}
\usepackage{psfrag}
\usepackage{verbatim}

\bibpunct{(}{)}{;}{a}{,}{,}

\numberwithin{equation}{section}
\newtheorem{thm}{Theorem}[section]
\newtheorem{lemma}{Lemma}[section]
\newtheorem{cor}{Corollary}[section]

\title{On The Drift Paradox in a Regime-Switching Model}
\author{William Felder\thanks{Department of Mathematics, 368 Kidder Hall, Oregon State University, Corvallis, OR.}
\and
Edward C. Waymire\thanks{Department of Mathematics, 368 Kidder Hall,  Oregon State University, Corvallis, OR. Corresponding author. \tt {waymire@math.oregonstate.edu}.}
}

\begin{document}

\maketitle

\begin{abstract}
This note is motivated by the article by  F. Lutscher, E. Pachepsky, and M.
Lewis (2005), The Effect of Dispersal Patterns on Stream Populations SIAM Rev.
Vol. 47 No. 4 pp. 749-772 on the drift paradox.   We consider the case of
a regime switching probabilistic model for a population of organisms living in a one dimensional environment  with drift towards an  absorbing boundary of the type introduced by
Lutscher et al. (2005).  In particular,  the two regimes consist of birth/death style demographics
governing the evolution of the immobile regime, and some form of
advection-dispersion governing the evolution of the mobile regime, together with a
regime-switching mechanism linking the two.  In the present note it is shown
for  the regime-switching model, and in contrast to the results
in the afore cited work, for any finite advection speed, no matter how
large, there is a finite critical length of the domain above which the population
can persist.
\end{abstract}

\section{Introduction and Preliminaries}
In the paper \cite{lutscher_etal_2005}, a benthos-drift model is presented, wherein a population of organisms existing in an environment subject to uni-directional flow is modeled as occupying two regimes: the benthic organisms are reproductive and non-mobile (these organisms are said to occupy regime 0 in our terminology), and the mobile organisms are non-reproductive (regime 1, here).  Switching from the non-mobile to the mobile regimes is taken to occur at some constant rate, here called $\lambda_0$, and switching from the mobile to the non-mobile regime also occurs with with constant rate, here called $\lambda_1$;  such models fall within a general class of so-called {\it regime or hybrid 
switching} models, see
\cite{yin}.
\par
It is argued in (\cite{lutscher_etal_2005}, p. 755) that, \lq\lq if the rate at which individuals move times the probability that they leave the domain during dispersal exceeds the population growth rate, then the population will go extinct\rq\rq.   
In particular, this provides a critical equation for length as a function of the growth rate and drift speed in the form
\begin{equation}
\label{LPLeqn}
r = \lambda_0\pi (\lambda _1,L,Y)
\end{equation}
where $L$ (\lq\lq length\rq\rq) is the location of an absorbing boundary, $Y=\{Y_t:t\ge 0\}$ is a stochastic process (to be described further) that describes the motion of individuals in the mobile regime, and the \lq\lq wash-out probability\rq\rq is given by 
\begin{equation}
\label{washprob}
\pi (\lambda_1, L, Y) = P(T_L < S_1),
\end{equation}
where $T_L$ is the first passage time to the absorbing boundary at $L$, and $S_1$, the regime switching time, is an exponentially distributed  random variable with parameter $\lambda_1$ 
representing the amount of time an individual will remain motile before settling back into the benthos.  In \cite{lutscher_etal_2005} this formula (\ref{washprob})  is obtained by means of a spectral analysis of another model (\ref{popdyn}) for population density $u$ that in turn is argued to yield (\ref{LPLeqn}),
where 
\begin{equation}
\label{popdyn}
{\partial u(x,t)\over\partial t} = ru(x,t) -\lambda u(x,t) + \int_0^L \lambda k(y,x)u(y,t)dy,
\end{equation}
for $r, \lambda > 0$, and a transition probability kernel $k(y,x)$.
\
We consider the underlying regime switching model for the mobile and immobile populations from 
(\cite{lutscher_etal_2005}, p. 755) together with this formula (\ref{LPLeqn}), in place of
(\ref{popdyn}),  as  the starting point for an alternative analysis of  the drift paradox.   
In particular, we will see that the critical length so-determined differs substantially from that
determined by (\ref{popdyn}).   
\\

The drift speed is determined by the following general assumption regarding the nature of the
advection-dispersion in the mobile regime.

\noindent{\bf The Drift Hypothesis (DH):}
{\it Assume $Y = \{Y_t: t\ge 0\}$ is a stochastic process starting at $Y_0 = 0$ and
having independent, stationary increments with $\mathbb{E}(Y_{t+s} - Y_s) = vt \quad  0<v<\infty$.    Moreover, assume  that  the paths of $Y$ are right-continuous without upward jumps.} 

\ 

The stochastic process $Y$ represents  the motion of individuals in the mobile phase.  
Such processes are said to be {\it spectrally negative 
L\'evy processes}. The point of the spectral negativity assumption is to exclude motile particles that can jump over $L$ without being absorbed there.  We note that such a process is non-explosive under the assumption of right-continuity
in (DH).
\\
The {\it L\'evy-It\^o decomposition} for this process is
\begin{equation}
Y_t=\tilde{v}t+\sqrt{D}B_t+J_t,
\end{equation}
where $\tilde{v}>0$ is deterministic drift to the right, $D\ge0$, $B_t$ is a standard Brownian motion, and $J_t$ is a L\'evy process, which jumps only to the left.
\\
With $v$ defined by (DH), it can be shown that
\begin{equation}
v=\tilde{v}+\int_0^\infty xn(dx),
\end{equation}
where  $n(\cdot)$ is the L\'evy measure associated with this process.\\
This quantity $v$ is called the {\it effective velocity},  and it accounts for both deterministic drift to the right, and average L\'evy drift.  Our assumption in (DH) that $v>0$, amounts to a constraint on the pair $\tilde{v},n(\cdot)$.
\\

The following simple lemma provides an alternative approach to the calculation of wash-out probabilities  for this model in terms of
the Laplace transform of the first passage time to the boundary evaluated at $\lambda_1$.

\begin{lemma}
\label{washtoLT}
 The wash-out probability is given by,
\begin{equation}
\pi (\lambda_1, L, Y) = \mathbb{E} e^{-\lambda_1 T_L}.
\end{equation}
\end{lemma}

\begin{proof}
\begin{equation}
\pi (\lambda_1, L, Y) = P(T_L < S_1) =\mathbb{E} (P(T_L < S_1 | T_L)) = \mathbb{E} (e^{-\lambda_1 T_L}).
\end{equation}
\end{proof}

This is especially relevant when viewed  in terms of the results obtained by
\cite{Pakes} that can be used to provide an explicit role of the parameters 
in connection with the wash-out probability in Lemma \ref{washtoLT}.  Namely, 
one has

\begin{equation}
\label{pakeseqn}
\mathbb{E}\left(e^{-sT_L}\right)=e^{L\eta(s)},
\end{equation}
where $\theta=\eta(s)$ is the unique solution to $\varphi(\theta)=-s$, and $\varphi(\theta)$ is the cumulant function of the process $Y$:
\begin{equation}
\varphi(\theta)=\left(\tilde{v}+\int_0^1xn(dx)\right)\theta-\frac{D\theta^2}{2}+\int_0^1(1-e^{-\theta x}-\theta x)n(dx)
+\int_1^\infty(1-e^{-\theta x})n(dx).
\end{equation}
\ 

In particular, using (\ref{pakeseqn})  it follows that
\begin{cor}
\begin{equation}
\mathbb{E} T_L=\frac{L}{v},
\end{equation}
where $v$ is the effective velocity, $v=\tilde{v}+\int_0^\infty xn(dx)$.
\end{cor}

\begin{lemma}
\label{lemma2}
The function $L \to \pi (\lambda_1, L, Y)$, for a fixed value of $\lambda_1$, is continuous and non-increasing, with $\pi (\lambda_1, 0, Y) = 1$ and $lim_{L \to \infty} \pi (\lambda_1, L, Y)=0$.
\end{lemma}

\begin{proof}
To see that $\pi (\lambda_1,0,Y) = 1$, note that $Y_0 = 0$ implies $T_0 = 0$ so that $\pi (\lambda_1,0,v) = P(0 < S_1) = 1$. That $\pi$ is a non-increasing function of $L$ follows easily from the fact that $L_1 < L_2$ implies $T_{L_1} < T_{L_2}$ so that $T_{L_2} < S_1$
implies  $T_{L_1} < S_1$ and we have $P(T_{L_2} < S_1) \le P(T_{L_1} < S_1).$
To prove that $lim_{L \to \infty} \pi (\lambda_1,L,Y) = 0$, first note that $T_L \to \infty$ as $L \to \infty$ by the non-explosivity of the underlying L\'evy process $Y$ under hypothesis (DH).   So the asserted limit
follows from Lemma \ref{washtoLT} and the dominated convergence theorem.
The continuity of $\pi$ as a function of $L$ is also evident from the formulae of \cite{Pakes}.
\end{proof}

In the case that $0 < \lambda_0 < r$, as noted in \cite{lutscher_etal_2005}, the population growth is obviously sufficient to prevent extinction without further need for analysis.  Thus we restict attention to  $\lambda_0 > r$ and obtain
the following main consequence of the previous lemmas. 

\begin{thm} Under (DH) with $\lambda_{0},\lambda_{1}, r>0$ and $r < \lambda_{0}$, there is an  $L = L_c$
 such that $r=\lambda_{0}\pi(\lambda_{1},L_c,Y)$.
\end{thm}
\begin{proof}
The asssertion follows immediately from Lemma \ref{lemma2} and the intermediate value theorem.
\end{proof}

\noindent{\bf Example 1:}
Suppose that the process governing the mobile phase is Brownian motion with drift $v$ and dispersion coefficient $D\ge 0$.  Then we have
\begin{equation}
P(T_L < S_1) = \hat{f}_{T_L}(\lambda_1) = \exp({vL \over D} - {\sqrt{v^2 + 2D\lambda_1}\over D}L)
\end{equation}
so that the critical equation (\ref{LPLeqn}) takes the form
\begin{equation}
\ln({r\over\lambda_0}) = ({v\over D} - {\sqrt{v^2+2D\lambda_1}\over D})L,
\end{equation}
so that
\begin{equation}
L = {\ln({\lambda_0\over r})\over 2\lambda_1}(\sqrt{1+{2D\lambda_1 \over v^2}}+1)v. 
\end{equation}
In particular it follows that in the limit  as $v\to\infty$ one obtains
\begin{equation}
\label{asympcrit}
L \approx {1\over\lambda_1}\ln({\lambda_0\over r})v.
\end{equation}
Observe that if $D=0$ then the asymptotic formula  (\ref{asympcrit})
is in fact the exact equation.   So this example also covers the case of deterministic motion in the mobile phase but, 
more to the point, also shows that large enough $v$ nuliifies the dispersion effect.

We see here that, although these results are based on both the underlying regime switching model introduced  in  (\cite{lutscher_etal_2005}, Sec 3.1), and the characterization of the critical length given by expression (3.8) there,  the conclusions above differ qualitatively from \cite{lutscher_etal_2005} concerning equation (\ref{LPLeqn}) when 
calculated in the context of (\ref{popdyn}).  
Specifically, for any finite advection velocity, no matter how large, there will be a length that satisfies (\ref{LPLeqn}), whereas in \cite{lutscher_etal_2005} it was found that there was a critical advection velocity $v_c$ such that $L \to \infty$ as $v\to v_c^{-}$, and above which {\it no length} will satisfy (\ref{LPLeqn}); see Figure \ref{criticallength} for comparison to (\cite{lutscher_etal_2005}, Figure 1).

\begin{figure}[htb] 
\includegraphics[scale=0.4]{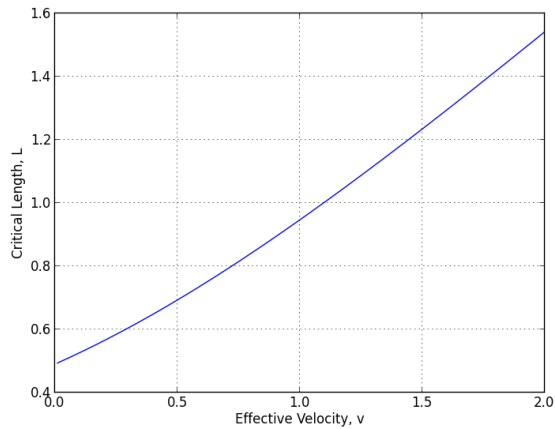}
\caption{Critical length as a function of drift velocity} 
\label{criticallength} 
\end{figure}

\section{Some Concluding Remarks}
While the main result of this paper is somewhat surprisingly at odds with the formula (4.6)  for the critical length 
obtained by
\cite{lutscher_etal_2005} as
a function of drift  speed, the resolution appears to be that the the regime-switching model and the 
model (\ref{popdyn}) are in fact quite different. 


In a still different but somewhat related framework to that of \cite{lutscher_etal_2005},  
\cite{kesten} obtained
 the existence of a finite critical drift speed for a system of branching Brownian motions with negative drift and absorbed at the origin.  While such results capture the spirit of the result in \cite{lutscher_etal_2005},  this involves non-linear production of offspring but {\it without} regime-switching.   The following simple result is provided to simply clarify the respective  roles of linear and non-linear production from a probabilistic perspective.  The case of non-linear production goes back to Henry McKean's  probabilistic representation of solutions to the {\it Kolmogorov-Petrovski-Piskunov-Fisher} (KPPF) reaction-diffusion
 equation\footnote{The KPPF equation is a reaction-diffusion independently introduced by Kolmogorov, Petrovski and
Piskunov on one hand, and by Fisher on the other.  It is sometimes simply referred to as KPP in the older mathematics literature.  However Fisher's considerations in populaton genetics
lead to the same equation, and more modern references either refer to FKPP or KPPF.} in \cite{mckean}.  While we expect that the description given below for the linear case is known, we are unable to find a reference.\footnote{Of course, in the case of a negative rate in the linear case, the probabilistic representation is also well known in terms of \lq particle killing\rq via the Feynman-Kac formula.} 
 
\begin{thm}
\label{clone}
 Let $A_rf(x) = {1\over 2}Df^{\prime\prime}(x) + vf^\prime(x) + rf(x)$ be the infintesimal generator of the transition operator $T_t^r$,
where $r > 0$, $D\ge 0$, $v\in {\mathbb{R}}$.   Also let $Y$ be a Markov process with infinitesimal
generator $A_0 = {1\over 2}D{\partial^2\over\partial x^2} + v{\partial\over\partial x}$ and transition operator $T_t^0$, and let
$N = \{N_t:t\ge 0\}$ be a Yule-Furry branching process with rate $r> 0$, which is independent of $Y$.   Then for 
$f\in {\mathcal D}_A$ one has that $M_t = N_tf(Y(t)) - \int_0^t A_rN_sf(Y_s)ds, t\ge 0,$ is 
a martingale. 
\end{thm}

\begin{proof} Let $\mathcal{F}_t = \sigma(N_s,Y_s:s \le t)$.  Then we have, for all $0 \le s \le t$,
\begin{eqnarray}
\mathbb{E}(M_t|\mathcal{F}_s) &=& \mathbb{E}(N_t|\mathcal{F}_s)\mathbb{E}(f(Y_t)|\mathcal{F}_s) - \int_0^s N_{s^\prime}f(Y_{s^\prime})ds^\prime - \int_s^t \mathbb{E}(N_{s^\prime}A_rf(Y_{s^\prime})|\mathcal{F}_s)ds^\prime\nonumber\\
		     &=& N_se^{r(t-s)}T_{t-s}^0f(Y_s) - \int_0^s A_rN_{s^\prime}f(Y_{s^\prime})ds^\prime - \int_s^t e^{r(s^\prime - s)}N_s({\partial T_{s^\prime - s}^0f(Y_{s^\prime})\over\partial s^\prime}\nonumber\\
		     & \  &+ rT_{s^\prime - s}^0f(Y_s))ds^\prime\nonumber\\
		     & =& N_se^{r(t-s)}T_{t-s}^0f(Y_s) - \int_0^s A_rN_{s^\prime}f(Y_{s^\prime})ds^\prime - (e^{r(t - s)}N_sT_{t-s}^0f(Y_s) - N_sf(Y_s))\nonumber\\
		     &=& N_sf(Y_s) - \int_0^s A_rN_{s^\prime}f(Y_{s^\prime})ds^\prime = M_s.
\end{eqnarray}
\end{proof}

It follows from  Theorem \ref{clone} that letting $Y^{(j)}  = Y, \ j = 1,\dots, N_t$,
\begin{equation}
u(x,t) = \mathbb{E}\sum_{j=1}^{N_t}f(Y^{(j)}_t)  \equiv \mathbb{E}_xN_tf(Y_t), \quad t\ge 0,
\end{equation}
solves 
\begin{equation}
{\partial u\over\partial t} = A_ru(x,t), \quad u(x,0^+) = f(x), \ x\in{\mathbb{R}}, t  > 0.
\end{equation}

In the case of KPPF the branching offspring evolve independently along 
paths distributed as their parent, whereas in the case of linear production
the offspring particles
are clones following precisely the same path as their parent.  This is an 
essential difference between the stochastic branching and (linearized) dynamic growth models.  
In the case of branching Brownian motions with negative drift and in the absence of regime
switching, \cite{kesten} shows that
at sufficiently high drift rate 
all particles will eventually be absorbed at the outlet regardless of $L$, 
namely $|v| > \sqrt{2Dr}$, while there is a chance to persist at lower drift speeds toward
the absorbing boundary.
For the linear production rates analyzed here and in \cite{lutscher_etal_2005},  because
of regime switching and cloning the eventual absorption is sure to occur for any positive 
drift speed only if $L$ is too small, i.e., it is always possible to find a (finite) length $L$ 
large enough for persistence under arbitrarily specified drift and regime switching rates.

\section*{Acknowledgments}  The  authors wish to thank Enrique Thomann
for helpful discussions about this problem,  and for calling our attention to
regime switching theory.

\bibliographystyle{plainnat}
\bibliography{DP.bib}

\begin{thebibliography}{5}
\providecommand{\natexlab}[1]{#1}
\providecommand{\url}[1]{\texttt{#1}}
\expandafter\ifx\csname urlstyle\endcsname\relax
  \providecommand{\doi}[1]{doi: #1}\else
  \providecommand{\doi}{doi: \begingroup \urlstyle{rm}\Url}\fi

\bibitem[Kesten(1978)]{kesten}
H.~Kesten.
\newblock Branching brownian motion with absorption.
\newblock \emph{Stochastic Processes and their Applications}, 7(1):\penalty0
  9--47, 1978.

\bibitem[Lutscher et~al.(2005)Lutscher, Pachepsky, and
  Lewis]{lutscher_etal_2005}
F.~Lutscher, E.~Pachepsky, and M.A. Lewis.
\newblock The effect of dispersal patterns on stream populations.
\newblock \emph{SIAM Journal of Applied Mathematics}, 65\penalty0 (4):\penalty0
  1305--1327, 2005.
\newblock Reprinted in {\it SIAM Review}, v. 47(4), (2005), 749-772.

\bibitem[McKean(1975)]{mckean}
H.~P. McKean.
\newblock Application of {B}rownian motion to the equation of
  {K}olmogorov-{P}etrovskii-{P}iskunov.
\newblock \emph{Comm. Pure Appl. Math.}, 28:\penalty0 323--331, 1975.

\bibitem[Pakes(1996)]{Pakes}
A.~G. Pakes.
\newblock A hitting time for {L}\'evy processes, with applications to dams and
  branching processes.
\newblock \emph{Ann. Fac. Sci. Toulouse Math. (6)}, 5\penalty0 (3):\penalty0
  521--544, 1996.

\bibitem[Yin and Zhu(2010)]{yin}
G.G. Yin and C.. Zhu.
\newblock \emph{{Hybrid switching diffusions: Properties and applications }}.
\newblock Springer, 2010.

\end{thebibliography}

\end{document}